\documentclass[reqno, a4paper]{amsart}
\usepackage{amssymb, mathdots}
\usepackage{mathabx}
\usepackage{mathrsfs}
\usepackage[utf8]{inputenc}
\usepackage{amsmath,  graphicx, tensor, stackrel}
\usepackage{tikz}
\usetikzlibrary{matrix, arrows.meta}
\usetikzlibrary{cd}

\usepackage{enumerate}
\usepackage{hyperref}
\DeclareSymbolFont{bbold}{U}{bbold}{m}{n}
\DeclareSymbolFontAlphabet{\mathbbold}{bbold}
\def\qmod#1#2{{\hbox{}^{\displaystyle{#1}}}\!\big/\!\hbox{}_{
\displaystyle{#2}}}

 \def\psp#1#2%
  {\mathop{}%
   \mathopen{\vphantom{#2}}^{#1}%
   \kern-\scriptspace%
    \hskip -0.3mm{#2}} 
 \def\psb#1#2%
  {\mathop{}%
   \mathopen{\vphantom{#2}}_{#1}%
   \kern-\scriptspace%
    \hskip -0.0mm{#2}} 
 \def\pscr#1#2#3%
  {\mathop{}%
   \mathopen{\vphantom{#3}}^{#1}_{#2}%
   \kern-\scriptspace%
    \hskip -0.3mm{#3}} 

\def\C{{\mathbb C}}

\def\N{{\mathbb N}}

\def\R{{\mathbb R}}

\def\textmap#1{\mathop{\vbox{\ialign{
                                  ##\crcr
      ${\scriptstyle\hfil\;\;#1\;\;\hfil}$\crcr
      \noalign{\kern 1pt\nointerlineskip}
      \rightarrowfill\crcr}}\;}}
\def\bigtextmap#1{\mathop{\vbox{\ialign{
                                  ##\crcr
      ${\hfil\;\;#1\;\;\hfil}$\crcr
      \noalign{\kern 1pt\nointerlineskip}
      \rightarrowfill\crcr}}\;}}
      
\newcommand{\cal}{\mathcal}
\def\textlmap#1{\mathop{\vbox{\ialign{
                                  ##\crcr
      ${\scriptstyle\hfil\;\;#1\;\;\hfil}$\crcr
      \noalign{\kern-1pt\nointerlineskip}
      \leftarrowfill\crcr}}\;}}

\def\rg{{\mathfrak r}}

\newtheorem{EX}{example}[section]
\theoremstyle{remark}
\newtheorem{ex}[EX]{Example} 
\newtheorem{exs}[EX]{Examples}

\newtheorem{sz}{Satz}[section]
\theoremstyle{remark}
\newtheorem{re}[sz]{Remark} 
\theoremstyle{plain}
\newtheorem{thry}[sz]{Theorem}
\newtheorem{pr}[sz]{Proposition}
\newtheorem{co}[sz]{Corollary}
\newtheorem{dt}[sz]{Definition}
\newtheorem{lm}[sz]{Lemma}

\def\GL{\mathrm {GL}}

\def\Tors{\mathrm{Tors}}

\def\id{ \mathrm{id}}
\def\im{\mathrm{im}}

\def\bpa{\bar\partial}

\newcommand\smvee{{\hskip -0.1ex \raise 0.2ex\hbox{$\scriptscriptstyle\vee$}}\hskip -0,3ex}

\def\Setminus{\;\setminus\;}

\def\inte{\mathrm{int}}

\def\edf{\coloneq}
\def\Setminus{\;\setminus\;}

\begin{document}

\title[An extension theorem for bundles]{An extension theorem for bundles with respect to   strictly pseudoconvex extensions}

\author{Andrei Teleman}
\address{Aix Marseille Univ, CNRS, I2M, Marseille, France.}
\email[Andrei Teleman]{andrei.teleman@univ-amu.fr}
 
\begin{abstract}

Let  ${\cal X}$ be a complex manifold and $X$ a compact complex manifold with boundary in ${\cal X}$. For a complex Lie group $G$ and a regularity class 
$$\rg\in \big\{{\cal C}^k|\ k\in\N\cup\{\infty\}\big\}\cup\big\{\Lambda^r_{\rm loc}|\ r\in (0,\infty)\big\} $$
we define the sheaf of groups ${\cal O}^{\rg\,G}_X$ on $X$ by
\begin{align*}
{\cal O}^{\rg\,G}_X(V)\edf \{u\in {\cal C}(V,G)|\ &u \hbox{ has regularity class $\rg$ on $V$},
\\ 
&u|_{V\cap\inte(X)} \hbox{ is holomorphic}\}.
\end{align*}
Let  $X_0\Subset Z_0\Subset {\cal X}$ be a strictly pseudoconvex extension in ${\cal X}$ \cite{HL}, and let $X\edf \bar X_0$, $Z\edf \bar Z_0$ be the corresponding compact manifolds with boundary in ${\cal X}$. Let $K\subset X_0$ be a compact set and ${\cal P}$   a holomorphic principal $G$-bundle of class $\rg$ (i.e. an ${\cal O}^{\rg\,G}_X$-torsor) on   $X\setminus K$.  Assuming that \vspace{1mm}
\\
 (H1) the given  strictly pseudoconvex  extension $X_0\Subset Z_0$ is non-critical, 
 \vspace{1mm}\\
 or that
 \vspace{1mm}\\
(H2) the underlying topological bundle of ${\cal P}$ extends to $Z\setminus K$,
 \vspace{1mm}\\ 
we prove that   ${\cal P}$ admits an extension to $Z\setminus K$. 

This gives a new proof and a new generalisation of Donaldson's extension problem stated in \cite{Do} and studied with different methods in \cite{Te}.
 
 \end{abstract}

 \subjclass[2020]{32L05, 32T15}

\maketitle

\tableofcontents

 \section*{Acknowledgements}
 I am grateful to Franc Forstnerič for explaining me his results on the Oka principle on complex manifolds with boundary and for a very useful exchange of mails concerning bundles of Hölder-Zygmund class $\Lambda^r$. I also thank Matei Toma for many stimulating discussions on the subject.

\section{Introduction}

\subsection{Motivation and scope}

Let $D\Subset\C^n$ be smoothly bounded  domain $\C^n$. Endowing the closure $\bar D$ with the sheaf of rings ${\cal A}^\infty_{\bar D}$ defined by
$$
{\cal A}^\infty_{\bar D}(V)=\{f\in {\cal C}^\infty(V,\C)|\ f|_{V\cap D} \hbox{ is holomorphic}\}
$$
for relatively open subsets $V\subset \bar D$, we obtain a locally ringed space $(\bar D,{\cal A}^\infty_{\bar D})$. 

In \cite{Do} S. Donaldson stated an interesting extension problem for bundles on complex manifolds with boundary which can be formulated as follows (see \cite[p. 102]{Do}, \cite[Remark 1.4]{Te}):
\newtheorem*{DonPb}{Extension Problem}
\begin{DonPb}
Suppose that $\bar D$ is strictly pseudoconvex and let ${\cal F}$ be a finite rank locally free  ${\cal A}^\infty_{\bar D}$-module on $\bar D$. Prove that ${\cal F}$ extends to a holomorphic bundle on an open neighbourhood of $\bar D$ in $\C^n$ in the following sense: there exists a  holomorphic bundle $E$ on  an open neighbourhood $U$ of $\bar D$ in $\C^n$ such that ${\cal F}$ is isomorphic to the ${\cal A}^\infty_{\bar D}$-module 
$$
V\mapsto \{f\in {\cal C}^\infty(V,E)|\ f|_{V\cap D} \hbox{ is holomorphic}\}.
$$
\end{DonPb}
If the underlying topological vector bundle of ${\cal F}$ is trivial, it follows by the classical Grauert's classification  theorem for holomorphic bundles on open Stein manifolds \cite{Gr}, that ${\cal F}$ is globally free on $\bar D$. Combining this result with the  
 Kobayashi-Hitchin correspondence for compact complex manifolds with boundary (\cite[Theorem 1']{Do}), Donaldson gives a natural identification between the quotient 
$$
\qmod{{\cal C}^\infty(\partial\bar D,\GL(p,\C))}{\im\big(A^\infty(\bar D,\GL(p,\C))\to {\cal C}^\infty(\partial\bar D,\GL(p,\C))}
$$
and the moduli space of boundary framed Hermite-Einstein unitary connections on topologically trivial vector bundles of rank $p$ over $\bar D$. This identification can be viewed as a higher dimensional version of a fundamental factorisation theorem in loop theory.

Donaldson gives an elegant proof of his extension problem  for the case $n=2$ in  \cite[Appendix A]{Do} mentioning that it ``is almost certainly true in general''.
In \cite{Te} we gave a positive answer and we obtained a broad generalisation of this problem using the following strategy: ${\cal F}$  can be identified with the sheaf of local sections in a ${\cal C}^\infty$-bundle vector bundle $F$ over $\bar D$ which are {\it formally} holomorphic (i.e. are solutions of the equation $\delta\sigma=0$) with respect to a {\it formally} integrable Dolbeault operator $\delta$ (semi-connection) on $F$ (see \cite[p. 306] {Te} for the terminology used in this paragraph). It suffices to prove that the pair $(F,\delta)$ admits an extension $(F',\delta')$ to an open neighbourhood $U$ of $\bar D$ in $\C^N$ with $\delta'$ integrable.  

More generally one can consider:

\begin{itemize}
\item An open submanifold  $X_0$ of a complex manifold $\Omega$  whose closure $X\edf\bar X_0$ is a submanifold with smooth, strictly pseudoconvex boundary.
\item A complex Lie group $G$, a principal bundle $G$-bundle $\Pi$ of class ${\cal C}^\infty$ on $\Omega$ and a formally integrable bundle almost complex structure $J$ on the restriction $P\edf\Pi|_{X}$.
\end{itemize}
 Then, according to \cite[Theorem 1.1]{Te}, $J$ extends to an integrable bundle almost complex structure $J'$ on  $\Pi|_{\Omega'}$, where $\Omega'$ is an open neighbourhood of $X$ in $\Omega$.  The proof, based on Zorn's Lemma and inspired by the proof of Hill-Nacinovich's collar neighbourhood theorem (\cite{HiNa}, \cite{HiNafix}) does not need compactness of $X$. 
 
 The goal of this article is a new generalisation of Donaldson's extension problem, which does require compactness of $X$, but, compared to \cite[Theorem 1.1]{Te}, has the following advantages:
 \begin{itemize}
 \item Gives control on the ``size''  of an open neighbourhood $\Omega'$ to which $J$ extends as an integrable bundle almost complex structure. 	
 \item Applies to bundles of Hölder-Zygmund class $\Lambda^r_{\rm loc}$ for any $r\in(0,\infty)$.
 \item It generalises to bundles  defined only on ``the inner side'' of the boundary $\partial X$, i.e. it applies to a bundle $\Pi$ on $\Omega\setminus K$ and a formally integrable bundle almost complex structure $J$ on $\Pi|_{X\setminus K}$, where $K$ is an arbitrary compact subset of $X_0$.
 \item Specialised to the classical case $G=\GL(p,\C)$, it yields an extension theorem for sheaves of modules which are locally free only around the boundary $\partial X$, but only coherent (in classical complex geometric sense) on $X_0$.
 \end{itemize}
 
 Note that the Hölder-Zygmund class $\Lambda^r$ plays a fundamental role in  the theory of elliptic systems (see for instance \cite[section 4.3.4]{Tr1}), the $\bar\partial$-equation and the $\bar\partial$-Neumann problem (\cite{GrSt}, \cite{BGS}) and also in the theory of the Bergman projection \cite{PhSt}. Bundles of  Hölder-Zygmund class $\Lambda^r$ on complaex manifolds with boundary intervene in the very recent article  of Forstnerič \cite{Fo} dedicated to the Oka principle for this class of  bundles and also in a new joint research project  \cite{TeTo} dedicated to moduli spaces of boundary framed torsion free sheaves on compact complex manifolds with boundary.
 
The formalism used in \cite{Te} is differential geometric: holomorphic principal bundles on complex manifolds (with boundary) are regarded as differentiable principal bundles endowed with (formally) integrable bundle almost complex structures. This formalism is well adapted when all considered objects are differentiable of class ${\cal C}^\infty$ (or more generally differentiable of class ${\cal C}^k$), but is not optimal when one deals with bundles of Hölder-Zygmund class. 

 Throughout this article we will adopt a different formalism, namely we will use the well known and increasingly popular equivalence between principal bundles and torsors. A topological (differentiable of class ${\cal C}^k$, holomorphic) principal $G$-bundle over a topological space (differentiable manifold, complex manifold) $X$ can be viewed as a torsor over the sheaf of groups  of locally defined continuous (respectively differentiable of class ${\cal C}^k$, holomorphic) $G$-valued maps on $X$.
 
 Similarly, ler $X$ be a complex manifold with boundary. For an arbitrary regularity class $\rg$ (e.g. ${\cal C}^k$ or $\Lambda^r_{\rm loc}$) we define the sheaf of groups ${\cal O}^{\rg\,G}_X$ of $X$ by
 $$
 {\cal O}^{\rg\,G}_X(V)\edf\{u\in {\cal C}(V,G)|\ u \hbox{ has regularity class $\rg$ on $V$},\ u|_{V\cap\inte(X)} \hbox{ is holomorphic}\}.
 $$
 If $G=\C$ we will of course omit the superscript $G$ and we obtain a locally ringed space $(X,{\cal O}^{\rg}_X)$. For $\rg={\cal C}^k$ and $G=\C$ one obtains the sheaf ${\cal A}^k_{\bar X}$ introduced in the vast literature dedicated to sheaf theory on complex manifolds with boundary (see e.g. \cite{Lei}, \cite{Heu}, \cite{Se}, \cite{DrFo}).
 
 The data of a locally free ${\cal O}^{\rg}_X$-module ${\cal F}$ of rank $p$ on $X$ is equivalent to the data of an ${\cal O}^{\rg\,\GL(p,\C)}_X$-torsor: the torsor associated with ${\cal F}$ is the sheaf 
 $$
 V\mapsto \mathrm{Iso}_{{\cal O}^\rg_V}({\cal O}^{\rg\oplus p}_V,{\cal F}|_V)
 $$ 
 of locally defined isomorphisms of ${\cal O}^{\rg}_X$-modules between the free ${\cal O}^{\rg}_X$-module  ${\cal O}^{\rg\oplus p}_X$ and ${\cal F}$. Donaldson's original extension problem, as we have formulated above, states that, for $\rg={\cal C}^\infty$ (and under the specified assumptions), any ${\cal O}^{\rg\,\GL(p,\C)}_{\bar D}$-torsor can be extended to an ${\cal O}^{\rg\,\GL(p,\C)}_{U}$-torsor on an open neighbourhood $U$ of $\bar D$ in $\C^n$.
 
 This shows that the formalism of torsors is very convenient for studying and generalizing such extension problems.   In the original problem one just replaces 
 \begin{itemize}
 \item $\bar D$ by a compact manifold with boundary $X$ in a complex manifold ${\cal X}$,
 \item $\GL(p,\C)$ by an arbitrary complex Lie group $G$, and 	
 \item the initial regularity class ${\cal C}^\infty$ considered in \cite{Do} by a more general regularity class $\rg$.
 \end{itemize} 
 
 Note that, in order to give a rigourous sense to the notion ``extension'' which intervenes in our generalised extension problems, we need a formal definition of the notion ``restriction'' in the framework of torsors. More precisely, if $\Omega$   is an open  (or compact with smooth boundary) neighbourhood of $X$ in ${\cal X}$, and ${\cal P}$ is an  ${\cal O}^{\rg\,G}_\Omega$-torsor, the sheaf theoretical restriction ${\cal P}|_X$ is an $({\cal O}^{\rg\,G}_\Omega)|_X$-torsor. The restriction we need is the canonically associated  ${\cal O}^{\rg\,G}_X$-torsor  obtained using the ``group sheaf change'' functor introduced in  section \ref{sheaf-change-section}.
 
 \subsection{Main results}
 
 Let ${\cal X}$ be a complex manifold and $X_0\Subset Z_0\Subset {\cal X}$  be strictly pseudoconvex extension in ${\cal X}$ (see \cite[Definition 2.6]{HL}), and let $X\edf \bar X_0$, $Z\edf \bar Z_0$  the corresponding compact complex manifolds with boundary in ${\cal X}$.  Let  $K\subset X_0$ be a compact set,   $G$ be a complex Lie group, and 
 $$\rg\in\big\{{\cal C}^k|\ k\in\N\cup\{\infty\}\big\}\cup\big\{\Lambda^r_{\rm loc}|\ r\in (0,\infty)\big\}$$
one of the regularity classes introduced in section \ref{section-bdls-Lambda}.

\begin{thry}\label{mainTh}
Let ${\cal P}$ be a holomorphic principal $G$-bundle  of class $\rg$ on $X\setminus K$. 
\begin{enumerate}
	\item Suppose that  $X_0\Subset Z_0$ is a non-critical strictly pseudoconvex extension in ${\cal X}$ (\cite[Definition 2.1]{HL}). Then ${\cal P}$ is the restriction  of a principal $G$-bundle of class $\rg$ on $Z\setminus K$. 
	\item Suppose that $X_0\Subset Z_0$ is just a strictly pseudoconvex extension in ${\cal X}$ (\cite[Definition 2.6]{HL}) and that the underlying topological bundle $P$ of ${\cal P}$ is isomorphic to the restriction of a topological bundle  $P'$ defined on $Z\setminus K$. Then ${\cal P}$ is the restriction  of a principal $G$-bundle of class $\rg$ on $Z\setminus K$ whose underlying topological bundle is isomorphic to $P'$. 
\end{enumerate}
\end{thry}
\begin{re} 
In particular, if the assumptions of (1) or (2) are satisfied, then ${\cal P}$ is   the restriction of 	a holomorphic principal $G$-bundle on the complex manifold (without boundary) $Z_0 \setminus K$.
\end{re}

\begin{re}
In the special case $K=\emptyset$ we obtain very simple extension properties for  a bundle ${\cal P}$ on $X$ defined on the closure $X$ of the smaller term  of a non-critical (or general) strictly pseudoconvex extension $X_0\Subset Z_0$.  	
\end{re}

In these statements we have used the notion ``restriction" explained formally in section \ref{section-bdls-Lambda} (see  Remark \ref{restrict-r-bdl-rem}, Definition \ref{restrict-r-bdl-def}).  
Taking $G=\GL(p,\C)$ and $K$ the  singularity set  $B({\cal F})$ of ${\cal F}$ (see \cite[p. 92]{GrRe}), we obtain:

\begin{co}
Let ${\cal F}$ be an ${\cal O}^{\rg}_{X}$-module on $X$	which is locally free of finite rank $p$ in a neighbourhood of $\partial X$ and whose restriction to $X_0$ is coherent. Suppose that 
\begin{tabular}{ll}
(H1)&\hspace{-3mm}The pseudoconvex extension $X_0\Subset Z_0$ is non-critical,
\\
&or that
\\
(H2)&\hspace{-3mm}The underlying topological bundle of the rank $p$  vector bundle defined by ${\cal F}$  \\ 
& on $X\setminus B({\cal F})$  extends to $Z\setminus B({\cal F})$. 
\end{tabular}

 Then ${\cal F}$ extends to a coherent  sheaf $\tilde {\cal F}$ on $Z_0$ with  $B(\tilde{\cal F})= B({\cal F})$.
\end{co}

For the proof of Theorem \ref{mainTh} we use a powerful well known  tool in complex geometry, namely the technique of convex bumps as explained for instance in \cite{HL}. In the presence of a non-critical  strictly pseudoconvex extension $X_0\Subset Z_0\Subset {\cal X}$ in ${\cal X}$, this technique allows one to pass from $X=\bar X_0$ to $Z=\bar Z_0$ by a finite number of convex  bumps  (see \cite[Lemma 2.2]{HL}). For a general strictly pseudoconvex extension one needs a sequence of special pseudoconvex bumps (see \cite[Corollary 2.8]{HL}).
Then we use  Propositions \ref{ExtBdlOnBump1}, \ref{ExtBdlOnBump2} which give  general extension properties for  bumps.    Let  $[D_-,U,D_+]$ be a special  pseudoconvex bump in ${\cal X}$ and ${\cal P}$  a  holomorphic principal $G$-bundle of class $\rg$ on $\bar D_-$. 

Assume that first that the bump is convex. Then ${\cal P}$ admits an extension ${\cal Q}$ to $\bar D_+$ which is trivial on $\bar U$. 

Second, assume that $[D_-,U,D_+]$ is a general special pseudoconvex bump and the underlying topological bundle $P$ of ${\cal P}$ admits an extension $P'$ to $D_+$. Then ${\cal P}$ admits an extension ${\cal Q}$ to $\bar D_+$ which is trivial on $\bar U$ and which is topologically isomorphic to $P'$. 
 
 We have similar statements for a bundle ${\cal P}$ defined only on $\bar D_-\setminus K$, where $K\subset D_-\setminus \bar U$ is a compact set.
 The main ingredient which comes in the proof of these propositions is the Oka principle for holomorphic bundles of class $\rg$ on strictly pseudoconvex compact manifolds with boundary in Stein manifolds (\cite{DrFo}, \cite{Fo}). We state the precise statement we need in Proposition \ref{Oka}. 
\section{Preliminaries}

\subsection{Torsor over a  sheaf of groups and principal bundles with structure sheaf}

Let $X$ be a topological space and ${\cal G}$ a sheaf of groups on $X$. 
\begin{dt}\label{def-torsor}
A right ${\cal G}$-torsor is  a sheaf of sets ${\cal T}$ on $X$ endowed with a right group sheaf action ${\cal T}\times{\cal G}\to{\cal T}$ satisfying:
\begin{enumerate}[(To1)]
\item \label{To1} For any open set $U\subset X$, the induced right action 
$$
 {\cal T}(U)\times {\cal G}(U)\to {\cal T}(U)
$$
is free and transitive.
\item \label{To2} Any point $x\in X$ has an open neighbourhood $U$ for which ${\cal T}(U)\ne\emptyset$. 
\end{enumerate}	

A right ${\cal G}$-torsor ${\cal T}$ will be called trivial if it admis a global section, i.e. if it is isomorphic to ${\cal G}$ as a right ${\cal G}$-torsor.
\end{dt}

\begin{re} Note that: 
\begin{enumerate}[(1)]
\item Condition (To\ref{To1}) is obviously satisfied when ${\cal T}(U)=\emptyset$.
\item  For any open set $U$ with ${\cal T}(U)\ne\emptyset$, the set ${\cal T}(U)$ becomes a right ${\cal G}(U)$-torsor in set theoretical sense.
\item Condition (To\ref{To2}) is equivalent to:  for any point $x\in X$, the stalk ${\cal T}_x$ of ${\cal T}$ at $x$ is non-empty. Taking into account (To\ref{To1}), it follows that the stalk  ${\cal T}_x$ of a right ${\cal G}$-torsor becomes a  right ${\cal G}_x$-torsor in set theoretical sense.
\end{enumerate}	
\end{re}

Using the formalism of étale spaces \cite[p. 111]{God} and denoting by $\mathscr{G}\to X$ the étale space of ${\cal G}$, we can equivalently define
\begin{dt}\label{def-torsor-etale}
A ${\cal G}$-torsor is an étale space $\mathscr{T}\to X$ with non-empty fibres over $X$ endowed with a continuos map
$$
\mathscr{T}\times_X\mathscr{G}\to \mathscr{T}
$$	
over $X$ such that 
\begin{enumerate}
\item For any $x\in X$ the induced map $\mathscr{T}_x\times\mathscr{G}_x\to\mathscr{T}_x$ between fibres is a free and transitive $\mathscr{G}_x$-action. 
\\
For $x\in X$ and $\tau$, $\tau'\in \mathscr{T}_x$ denote by  $\tau^{-1}\tau'$ the unique element $\gamma\in\mathscr{G}_x$ satisfying $\tau\gamma=\tau'$. 
\item The map 
$$\mathscr{T}\times_X\mathscr{T}\to \mathscr{G},\ (\tau,\tau')\mapsto \tau^{-1}\tau'$$
is continuous.
\end{enumerate}
\end{dt}

Let now $X$ be a topological space and ${\cal C}^G_X$ be the sheaf of continuous locally defined $G$-valued  maps on $X$. ${\cal C}^G_X$ is naturally a sheaf of groups on $X$. 

\begin{dt}\label{calG-bdl}
Let ${\cal G}\subset {\cal C}^G_X$ be a subsheaf of subgroups of ${\cal C}^G_X$. A principal ${\cal G}$-bundle on $X$ is a right ${\cal G}$-torsor ${\cal T}$ on $X$ in the sense of Definition \ref{def-torsor}. A  principal ${\cal G}$-bundle ${\cal T}$ on $X$ will be called trivial if it is trivial as a ${\cal G}$-torsor. 

\end{dt}

Let ${\cal T}$ be a principal ${\cal G}$-bundle on $X$ in the sense of Definition \ref{calG-bdl}. For a local section $\sigma$ of ${\cal T}$ we will denote by $U_\sigma$ its domain, i.e. the open set of $X$ for which $\sigma\in {\cal T}(U_\sigma)$. 
For local sections $\sigma$, $\tau$  of ${\cal T}$ let $g_{\tau\sigma}\in {\cal G}(U_\sigma\cap U_\tau)\subset {\cal C}^G_X(U_\sigma\cap U_\tau)$  be defined by the condition
$$
\sigma|_{U_\sigma\cap U_\tau}=\tau|_{U_\sigma\cap U_\tau} g_{\tau\sigma}.
$$
For local sections $\theta$, $\tau$, $\sigma$ we have the identity
$
g_{\theta\tau}g_{\tau\sigma}=g_{\theta\sigma} \hbox { on }U_\sigma\cap U_\tau\cap U_\theta$,
so the system 
\begin{equation}\label{cocyle}
(g_{\tau\sigma})_{\substack{\sigma,\;\tau\;\rm  local\\ \rm sect.\, of\, {\cal T} } }	
\end{equation}
\vspace{-1mm}\\
is a ${\cal G}$-valued Cech  1-cocycle for the open cover 
$(U_\sigma)_{\substack{\sigma\;\rm  local\\ \rm sect.\, of\, {\cal T} } }$
 of $X$.
The {\it total space} of ${\cal T}$ is the right $G$-space $P_{\cal T}$ over $X$ which is associated with the cocycle  (\ref{cocyle}), i.e. 
$$
P_{\cal T}\edf \bigg(\coprod_
{\substack{\sigma\;\rm  local\\ \rm sect.\ of\ {\cal T} } }   \{\sigma\}\times U_\sigma\times G\bigg)/ {\sim}_{\cal T} \ ,
$$
where $\sim_{\cal T}$ is the equivalence relation generated by the pairs
$$
\big((\sigma,x,\gamma),(\tau,x,g_{\tau\sigma}\gamma)\big),\  x\in U_\sigma\cap U_\tau.
$$

\begin{exs}\label{ex-torsors-bundles}
Let $X$ be a topological space. The data of a topological principal $G$-bundle on $X$ in the classical sense is equivalent to the data of a principal ${\cal C}^G_X$-bundle in the sense of Definition \ref{calG-bdl}.   

Let $X$ be a differentiable (complex) manifold, ${\cal C}^{k\,G}_X$ (\;${\cal O}_X^G$\;) the sheaf of locally defined $G$-valued maps of class ${\cal C}^k$ (respectively holomorphic $G$-valued maps, where $G$ is a complex Lie group).  The data of a principal $G$-bundle of class ${\cal C}^k$ (of a holomorphic principal $G$-bundle) on $X$ in the classical sense is equivalent to the data of a principal ${\cal C}^{k\,G}_X$-bundle (respectively ${\cal O}^G_X$-bundle) in the sense of  Definition \ref{calG-bdl}.

In each case we have an equivalence of groupoids given by the following functors:
\begin{itemize}
	\item The functor which assigns to a (topological, of class ${\cal C}^k$, holomorphic) principal $G$-bundle  in the classical sense its sheaf of local sections. 
	\item The functor which assigns to a principal ${\cal C}^G_X$ (${\cal C}^{k\,G}_X$, ${\cal O}_X^G$)-bundle ${\cal T}$ in the sense of Definition \ref{calG-bdl} the principal $G$-bundle $P_{\cal T}\to X$ endowed with its canonical topological (differentiable of class ${\cal C}^k$, holomorphic) structure.
\end{itemize}

\end{exs}

We are especially interested in principal ${\cal G}$-bundles, where the structure sheaf ${\cal G}$ is obtained from a sheaf of locally defined $G$-valued maps on the base by imposing a  condition which depends on the point; in these cases we do not have a bundle theory ``in the classical sense". 

\begin{ex} Let $X$ be a topological space, $F\subset X$ a closed set, and ${\cal C}^G_{X,F}$ the sheaf of locally defined $G$-valued continuous maps on $X$ which take the value $e_G$ (the unit element of $G$) on $F$. The data of a principal ${\cal C}^G_{X,F}$-bundle on $X$ is equivalent to the data of a topological principal $G$-bundle on $X$ endowed with a trivialisation (or, equivalently, a section) of its restriction to $F$.
	
\end{ex}

\subsection{Group sheaf change and pull-backs} \label{sheaf-change-section}

Recall first that if $e:G\to G'$ is a group morphism and $T\times G\to T$ a $G$-torsor in the set theoretical sense, then
\begin{equation}\label{group-ext}
T\times_e G'\edf \qmod{T\times G'}{G}
\end{equation}
is naturally a $G'$-torsor. In formula (\ref{group-ext}) $G$ acts freely from the left on $T\times G'$ by 
$$g(t,g')\edf (tg^{-1},e(g)g').$$
Similarly, let $X$ be a topological space, $\eta:{\cal G}\to{\cal G}'$  a morphism of   group sheaves on $X$ and ${\cal T}$ a right ${\cal G}$-torsor. The formula 
\begin{equation}\label{presheaf}
X\stackrel{\scriptscriptstyle\rm open}\supset U\mapsto {\cal T}(U)\times_{\eta_U}{\cal G}'(U)\end{equation}
defines a presheaf of sets on $X$. Note that a section $\tau\in {\cal T}(U)$ defines an isomorphism between the restriction of this presheaf to $U$ and the restriction ${\cal G}'|_U$, in particular the former restriction is already a sheaf. The   sheaf associated with the presheaf (\ref{presheaf}) will be denoted ${\cal T}\times_\eta{\cal G}'$. 
Since inductive limits with respect to filtered posets commute with Cartesian products and quotients by group actions, it follows that the stalk at $x$ of this sheaf is
$$
\big({\cal T}\times_{\eta}{\cal G}'\big)_x={\cal T}_x\times_{\eta_x}{\cal G}'_x.
$$
Endowed with the obvious right ${\cal G}'$-action, ${\cal T}\times_\eta{\cal G}'$ is a ${\cal G}'$-torsor on $X$, which will be called the ${\cal G}'$ torsor associated with ${\cal T}$ by group sheaf change via $\eta$. The assignment 
$$
{\cal T}\mapsto {\cal T}\times_\eta{\cal G}'
$$
defines a functor 
$$F_\eta:\Tors_{\cal G}\to \Tors_{{\cal G}'}$$
 from the groupoid of right ${\cal G}$-torsors to the groupoid of right ${\cal G}'$-torsors.

\begin{ex}
Let $X$ be a complex manifold and $G$ a complex Lie group,  and  $\iota:{\cal O}^G_X\hookrightarrow{\cal C}^{G}_X$, $\iota^k:{\cal O}^G_X\hookrightarrow{\cal C}^{k\, G}_X$, be the obvious group sheaf monomorphisms (see Example \ref{ex-torsors-bundles}). 
Via the groupoid equivalences explained in Example \ref{ex-torsors-bundles}, the functors  
$$F_\iota:\Tors_{{\cal O}^G_X}\to \Tors_{{\cal C}^G_X},\ F_{\iota^k}:\Tors_{{\cal O}^G_X}\to \Tors_{{\cal C}^{kG_X}}$$
 correspond to the well known functors which assign to a holomorphic principal bundle on $X$ its underlying topological   bundle, respectively its underlying differentiable bundle of class ${\cal C}^k$.
\end{ex}

Let $f:X\to X'$ be a continuous map between topological spaces, ${\cal G}'$ a sheaf of groups on $X'$, and ${\cal T}'$ a ${\cal G}'$-torsor. The inverse image $f^{-1}({\cal T}')$ is naturally an $f^{-1}({\cal G}')$-torsor on $X$.  This follows using the equivalent Definition \ref{def-torsor-etale} and the well known behaviour of the étale space functor under inverse image of sheaves (see for instance \cite[p. 121]{God}, \cite[section II.1B]{De}). 

On the other hand, in the classical theory of (topological, differentiable, holomorphic) principal bundles we have a different notion of pull back. For instance, a the pullback $f^*(P')$ of a topological principal $G$-bundle $P'$ on $X'$ is a topological principal $G$-bundle on $X$, whose associated sheaf of local sections is a ${\cal C}^G_X$-torsor, not an $f^{-1}({\cal C}^G_{X'})$-torsor.

Note that, in the presence of a continuous map $f:X\to X'$, we have a canonical $f$-map of sheaves of groups 
$$c_f: {\cal C}^G_{X'}\to {\cal C}^G_X$$
(see  \cite[\href{https://stacks.math.columbia.edu/tag/008J}{Tag 008J}]{St})  which is induced by family of composition maps
$$
{\cal C}(U',G)\to {\cal C}(f^{-1}(U'),G),\ \alpha\mapsto \alpha\circ f
$$   
associated with open sets   $U'\subset X'$. Recall (see \cite[\href{https://stacks.math.columbia.edu/tag/008K}{Tag 008K}]{St}) that $c_f$ can be viewed as a group sheaf morphism ${\cal C}^G_{X'}\to f_*({\cal C}^G_X)$, and also as a group sheaf morphism $f^{-1}({\cal C}^G_{X'})\to {\cal C}^G_{X}$.
\begin{re} Let $f:X\to X'$ be a continuous map between topological spaces, and let ${\cal G}\subset {\cal C}^G_X$, ${\cal G}'\subset{\cal C}^G_{X'}$ be subsheaves of subgroups on $X$, respectively $X'$.  The following conditions are equivalent:
\begin{enumerate}[(C1)]
\item \label{C1} For any open set $U'\subset X$ and any $\alpha\in {\cal G}'(U')$, we have $\alpha\circ f\in {\cal G}(f^{-1}(U'))$.
\item \label{C2} The sheaf morphism ${\cal C}^G_{X'}\to f_*({\cal C}^G_X)$ induced by $c_f$ maps the subsheaf ${\cal G}'$ of  ${\cal C}^G_{X'}$ into $f_*({\cal G})$.
\item 	\label{C3} The sheaf morphism $f^{-1}({\cal C}^G_{X'})\to {\cal C}^G_{X}$ induced by $c_f$ maps the subsheaf $f^{-1}({\cal G}')$ of  $f^{-1}({\cal C}^G_{X'})$ into ${\cal G}$.
\end{enumerate}	
\end{re}

\begin{dt}\label{pull-back-bundle}
If one of the equivalent conditions (C\ref{C1})-(C\ref{C3}) holds, we will say that   the compatibility condition $C({\cal G},{\cal G}',f)$ is satisfied, and we let   $\eta_f:f^{-1}({\cal G}')\to {\cal G}$ denote the group sheaf morphism induced by $c_f$.

Suppose this is the case, and let ${\cal T}'$ be a principal ${\cal G}'$-bundle on $X'$ in the sense of Definition \ref{calG-bdl}. We define the pull back ${\cal G}$-bundle $f^*({\cal T})$ of  ${\cal T}$ via $f$ by
$$
f^*({\cal T})\edf f^{-1}({\cal T})\times_{\eta_f}{\cal G}
$$	
(see section \ref{sheaf-change-section}).
\end{dt}

\begin{ex}\label{compat-condit-ex1}
Let $k\in\N\cup\{\infty\}$, $f:X\to X'$  a differentiable map of class ${\cal C}^k$ (a holomorphic map)  between differentiable (complex) manifolds, and let ${\cal T}$ be a principal ${\cal C}^{k\,G}_{X'}$ (respectively ${\cal O}^G_{X'}$-bundle) on $X'$ in the sense of Definition \ref{calG-bdl}.

The compatibility condition $C({\cal C}^{k\,G}_{X},{\cal C}^{k\,G}_{X'},f)$ (respectively $C({\cal O}^{G}_{X},{\cal O}^{G}_{X'},f)$) is obviously satisfied, so that the pull-back ${\cal C}^{k\,G}_{X}$-bundle (respectively the pull-back ${\cal O}^{G}_{X}$-bundle) $f^*({\cal T})$ in the sense of Definition \ref{pull-back-bundle} is defined. It corresponds to the  differentiable (holomorphic) pull back bundle $f^*({\cal P}_{\cal T})$ in the classical sense. 

\end{ex}

Whereas the compatibility conditions intervening in  Example \ref{compat-condit-ex1}  are obviously satisfied, the problem becomes more delicate in the  situation considered in the following section.
\subsection{Bundles of class \texorpdfstring{${\cal C}^k$}{1}, \texorpdfstring{$\Lambda^r_{\rm loc}$}{2} on manifolds with boundary}
\label{section-bdls-Lambda}

 We will use the following terminology:

\begin{dt}\label{XincalX}
Let ${\cal X}$ be a differentiable manifold (without boundary). A manifold with  boundary in ${\cal X}$ is an embedded submanifold with boundary  $X\subset {\cal X}$ of the same dimension. \end{dt}

Therefore, if $X$ is a  manifold with boundary in ${\cal X}$, then its interior as a manifold with boundary coincides with its topological interior in ${\cal X}$ denoted $\inte(X)$, so it is an open  submanifold of ${\cal X}$; its boundary $\partial X$  is a real hypersurface of ${\cal X}$ and $X\subset \overline{\inte(X)}$. We have equality $X=\overline{\inte(X)}$ if and only if $X$ is closed in ${\cal X}$.
\begin{re}\label{subman-with-bd}
Let  $X$ be a  manifold with  boundary in ${\cal X}$.
\begin{enumerate}
	\item\label{1st} Any relatively open subset of $X$ is also a manifold with  boundary in ${\cal X}$. More precisely, let ${\cal U}\subset {\cal X}$ an open subset.
Then $X\cap {\cal U}$ is submanifold with boundary in ${\cal U}$, $\partial(X\cap {\cal U})=\partial X\cap {\cal U}$, and $\inte(X\cap {\cal U})=\inte(X)\cap {\cal U}$.
\item \label{2nd} Suppose that $X$ is closed in ${\cal X}$. Then the complement $\complement_{\cal X}(\inte(X))$ is also a  manifold with boundary in ${\cal X}$, and 
$$\partial(\complement_{\cal X}(\inte(X)))=\partial X,\ \inte(\complement_{\cal X}(\inte(X)))=\complement_{\cal X}(X).$$
\end{enumerate}    	
\end{re}

Let $X$ be an $m$-dimensional manifold with boundary in ${\cal X}$, $Y$ an $n$-dimensional manifold with boundary in ${\cal Y}$ and $k\in\N\cup\{\infty\}$. We recall that a map $f:X\to Y$ is differentiable of class ${\cal C}^k$ if it is induced by a  map $\tilde f:{\cal V}\to {\cal Y}$ of class ${\cal C}^k$ with $\tilde f(X)\subset Y$, where ${\cal V}$ is an open neighbourhood of $X$ in ${\cal X}$.

For a Lie group $G$, a manifold with boundary $X$ in ${\cal X}$ and $k\in \N\cup\{\infty\}$ let ${\cal C}^{k\,G}_X$ be the sheaf of groups defined by 
$$
 U \mapsto {\cal C}^{k\,G}_X(U)={\cal C}^k(U,G)
$$ 
for relatively open subsets $U\subset X$.

\begin{dt}\label{Ck-Bundle}
Let $X$ be a manifold with boundary in ${\cal X}$, $k\in\N\cup\{\infty\}$, and $G$  a  Lie group. A  principal $G$-bundle of  class  ${\cal C}^k$ on $X$ is a principal 	${\cal C}^{k\,G}_X$-bundle on $X$ in the sense of Definition \ref{calG-bdl}.	
\end{dt}

Pull back bundles via maps of class ${\cal C}^k$ are also well defined in the framework of manifolds with boundary (compare with Example \ref{compat-condit-ex1}): 

\begin{re} Let $X$, $Y$ be manifolds with boundary in ${\cal X}$, respectively ${\cal Y}$ and $f:X\to Y$ a differentiable map of class ${\cal C}^k$. Then the compatibility condition $C({\cal C}^{k\,G}_X,{\cal C}^{k\,G}_Y,f)$ introduced in Definition \ref{pull-back-bundle} is satisfied. In particular, for any principal $G$-bundle ${\cal T}$ of  class  ${\cal C}^k$ on $Y$, we have a well defined   pull-back principal $G$-bundle  $f^*({\cal T})$  of  class  ${\cal C}^k$  on $X$.
\end{re}

We refer to \cite[Section V.4]{St}, \cite[Section 3.2]{Gon} for the definition and fundamental properties of the Zygmund spaces $\Lambda^r(\R^n)$.
Let $r\in(0,\infty)$. We recall \cite[Definition 3.9]{Gon} that, for a closed set $F\subset\R^n$, the space $\Lambda^r(F,\R^s)$ is defined by 
\begin{equation}\label{DefLambda(F)}
\Lambda^r(F,\R^s)\edf\{u|_F|\ u\in \Lambda^r(\R^n,\R^s)\}
\end{equation}
endowed with the Banach space structure induced via the obvious identification
$$\qmod{\Lambda^r(\R^n,\R^s)}{\{u\in \Lambda^r(\R^n,\R^s)|\ u|_F=0\}}\textmap{\simeq}\Lambda^r(F,\R^s)
$$
by the quotient norm on left hand space.

Let $X$ be a manifold with boundary in ${\cal X}$, $u:X\to \R^s$ a map,  $x\in X$ and let
$${\cal X}\stackrel{\scriptscriptstyle\rm open}\supset V\textmap{h}U\stackrel{\scriptscriptstyle\rm open}\subset\R^n$$
a chart of ${\cal X}$ defined on an open neighbourhood of $x$ in ${\cal X}$. We'll say that that $u$ is locally of class $\Lambda^r$ at $x$ if there exists a compact neighbourhood $K$ of $x$ in $X$ which is contained in $V$ such that the composition $u\circ h^{-1}|_{h(K)}$ belongs to $\Lambda^r(h(K),\R^s)$ in the sense defined above.

Note that this condition is independent of $h$. This follows from the following composition lemma:
\begin{lm}\label{composition-with-smooth-from-right}
Let $\theta:U\to V$ be a  map of class ${\cal C}^\infty$ between  open sets $U\subset\R^m$, $V\subset\R^n$. Let $K\subset U$, $L\subset V$ be compact subsets  such that   $\theta(K)\subset L$. Let $u\in\Lambda^r(L,\R^s)$. Then $u\circ\theta\in\Lambda^r(K,\R^s)$.  \end{lm}

\begin{proof} Let $\iota_V:V\hookrightarrow\R^n$ the inclusion map and let $\theta':U\to \R^n$ be a differentiable map which coincides with $\iota_V\circ \theta$ on a neighbourhood of $K$ and has compact support contained in $U$. Let $\tilde\theta:\R^m\to\R^n$ be the obvious extension of $\theta'$ (which is obviously differentiable), and let $\tilde u\in \Lambda^r(\R^n,\R^s)$ be an extension of $u$.  Such an extension exists by the definition formula (\ref{DefLambda(F)}). Then $\tilde u\circ \tilde \theta:\R^m\to\R^s$ is an extension of $u\circ\theta$.  It is enough to note that composition from the right with a compactly supported differentiable map $\R^m\to\R^n$ defines a bounded operator 
$$\Lambda^r(\R^n,\R^s)\to \Lambda^r(\R^m,\R^s).$$
A more general  composition result for the Hölder-Zygmund spaces is \cite[Lemma 3.1]{GaGo}.
\end{proof}

We will say that $f$ is locally of class $\Lambda^r$ if it is locally of  this class at any point $x\in M$. We denote by $\Lambda^r_{\rm loc}(X,\R^s)$ the space of maps $u:X\to\R^s$ which are locally of  class $\Lambda^r$.

\begin{pr}\label{map-X-Y}
Let $X$, $Y$ be manifolds with boundary in ${\cal X}$, respectively ${\cal Y}$ and $f:X\to Y$ a  map of class ${\cal C}^\infty$.  For any $u\in \Lambda^r_{\rm loc}(Y,\R^s)$ the composition $u\circ f: X\to \R^s$ belongs to $\Lambda^r_{\rm loc}(X,\R^s)$.
\end{pr} 
\begin{proof}

We know that $f$  is the restriction to $X$ of a differentiable map $\tilde f:{\cal V}\to {\cal Y}$, where ${\cal V}$ is an open neighbourhood of $X$ in ${\cal X}$. 
Let $x\in X$ and let
$${\cal Y}\stackrel{\scriptscriptstyle\rm open}\supset W\textmap{\chi}U\stackrel{\scriptscriptstyle\rm open}\subset\R^n$$
be a chart of ${\cal Y}$ defined on an open neighbourhood $W$ of $f(x)$ in ${\cal Y}$. Let $L$ be a compact neighbourhood of $f(x)$ in $Y$ such that $L\subset W$ and 
$$u\circ \chi^{-1}|_{\chi(L)}\in \Lambda^r(\chi(L),\R^s).$$
$f^{-1}(L)$ is a {\it closed} neighbourhood of $x$ in $X$, $f^{-1}(L)\subset \tilde f^{-1}(L)\subset \tilde f^{-1}(W)$, and $\tilde f^{-1}(W)$  is an open neighbourhood of $x$ in ${\cal V}$.

Let
$${\cal V}\stackrel{\scriptscriptstyle\rm open}\supset V\textmap{h}S\stackrel{\scriptscriptstyle\rm open}\subset\R^m$$
be a chart of ${\cal V}$ defined on an open neighbourhood $V$ of $x$ in ${\cal V}$ which is contained in $\tilde f^{-1}(W)$. Let $K$ be a compact neighbourhood of $x$ in $X$ which is contained in $f^{-1}(L)$ and also in $V\cap X$. 

We have 
$$
u\circ f\circ h^{-1}|_{h(K)}=u\circ \tilde f\circ h^{-1}|_{h(K)}= (u\circ \chi^{-1})\circ(\chi\circ  \tilde f\circ h^{-1})|_{h(K)}.
$$
The composition
$$
\chi\circ  \tilde f\circ h^{-1}:S\to U
$$
is a differentiable map between open subsets of $\R^m$, respectively $\R^n$ which maps the compact $h(K)$ into $\chi(L)$. By Lemma \ref{composition-with-smooth-from-right} it follows that $(u\circ f)\circ h^{-1}|_{h(K)}\in \Lambda^r(h(K),\R^s)$, which shows that $u\circ f$ is locally of class $\Lambda^r$ at $x$.
\end{proof}

\begin{re}\label{restr-Y-X-M}
Using the composition Lemma \cite[Lemma 2.3]{Fo}, one can coherently define the set $\Lambda^r_{\rm loc}(X,M)$ of maps $X\to M$ which are locally of class $\Lambda^r$ for any manifold with boundary in ${\cal X}$ and any differentiable target manifold $M$. Proposition \ref{map-X-Y} obviously extends to $M$-valued maps.
\end{re}

For a Lie group $G$, a manifold with boundary $X$ in ${\cal X}$ and $r>0$ let ${\cal L}^{r\,G}_X$ be the sheaf of groups defined by 
$$
 U \mapsto {\cal L}^{r\,G}_X(U)=\Lambda^r_{\rm loc}(U,G)
$$ 
for relatively open subsets $U\subset X$.

\begin{dt}\label{Lambda-r-Bundle}
Let $G$ be a  Lie group. A  principal $G$-bundle of  class  $\Lambda^r_{\rm loc}$ on $X$ is a principal 	${\cal L}^{r\,G}_X$-bundle on $X$ in the sense of Definition \ref{calG-bdl}.	
\end{dt}

\begin{pr}\label{pull-back-Lambda-r}
Let $X$, $Y$ be manifolds with boundary in ${\cal X}$, respectively ${\cal Y}$ and $f:X\to Y$ a  map of class ${\cal C}^\infty$. Then the compatibility condition $C({\cal L}^{r\,G}_X,{\cal L}^{r\,G}_Y,f)$ introduced in Definition \ref{pull-back-bundle} is satisfied. In particular, for any principal $G$-bundle ${\cal T}$ of  class  $\Lambda^r_{\rm loc}$ on $Y$, we have a well defined   pull-back principal $G$-bundle  $f^*({\cal T})$  of  class  $\Lambda^r_{\rm loc}$  on $X$.
\end{pr}
\begin{proof}
It suffices to note that, for any relatively open set $V\subset Y$, $V$ is also a manifold with boundary in ${\cal Y}$,   $f^{-1}(V)$ is also a manifold with boundary in ${\cal X}$, and $f$ induces a smooth map $f^{-1}(V)\to V$. The claim follows by Proposition \ref{map-X-Y} and Remark \ref{restr-Y-X-M}.
\end{proof}

In the special case when $X$, $Y$ are manifolds with boundary in the same ${\cal X}$ and $X\subset Y$, the pull-back bundle  $\iota^*({\cal T})$ via the  inclusion map $\iota:X\hookrightarrow Y$ will be called the restriction of ${\cal T}$ to $X$.
\vspace{2mm}

Similarly, in the complex geometric framework, in order to avoid considering the delicate collar problem for abstract complex manifolds with boundary (see \cite{Ca}, \cite{Hi}, \cite{HiNa}-\cite{HiNafix}), we will only consider complex manifolds with boundary which are already embedded in a complex manifold (without boundary) of the same dimension. More precisely 
\begin{dt}\label{XcomplexincalX}
Let ${\cal X}$ be a complex manifold. A complex manifold with boundary in ${\cal X}$ is a manifold with boundary in $X$ in the sense of Definition \ref{XincalX}.
\end{dt}

Therefore, if $X$  is a complex manifold with boundary in ${\cal X}$, then its interior $\inte(X)$ is an open complex submanifold of ${\cal X}$, and its boundary $\partial X$ is a real hypersurface in  ${\cal X}$, so it comes with a canonical CR-structure.

Let ${\cal X}$ be a complex manifold and $X$  a complex manifold with boundary in ${\cal X}$ in the sense of Definition \ref{XcomplexincalX}. Let 
$$\rg\in \big\{{\cal C}^k|\ k\in\N\cup\{\infty\}\big\}\cup\big\{\Lambda^r_{\rm loc}|\ r\in (0,\infty)\big\} $$
be one of the regularity classes introduced above  for maps defined on open sets in manifolds with boundary.

We define the sheaf of $\C$-algebras ${\cal O}^{\rg}_X$ and the sheaf of groups   ${\cal O}^{\rg\,G}_X$ on $X$ as follows:
\begin{itemize}
\item For $\rg={\cal C}^k$, we   put
\begin{equation}
\begin{split}
{\cal O}^{\rg}_X(V)\edf & \big\{f\in {\cal C}^k(V,\C)|\ f|_{V\cap\inte(X)} \hbox{ is holomorphic}\big\},\\
{\cal O}^{\rg\,G}_X(V)\edf &\big\{f\in {\cal C}^k(V,G)|\ f|_{V\cap\inte(X)} \hbox{ is holomorphic}\big\},
\end{split}
\end{equation}
\item for $\rg =\Lambda^r_{\rm loc}$, we put   %
\begin{equation}
\begin{split}
{\cal O}^{\rg}_X(V)\edf & \big\{f\in \Lambda^r_{\rm loc}(V,\C)|\ f|_{V\cap\inte(X)} \hbox{ is holomorphic}\big\}\\
{\cal O}^{\rg\,G}_X(V)\edf & \big\{f\in \Lambda^r_{\rm loc}(V,G)|\ f|_{V\cap\inte(X)} \hbox{ is holomorphic}\big\}
\end{split}
\end{equation}	
\end{itemize}
for relatively open sets $V\subset X$. Note that for $\rg={\cal C}^k$ the notation ${\cal A}^k_X$ is frequently used in the literature (\cite{Se}, \cite{DrFo}, \cite{Lei}, \cite{Heu}).
 
 \begin{dt}\label{OrG-bundle}
 Let $G$ be a complex Lie group and 
 $$\rg\in\big\{{\cal C}^k|\ k\in\N\cup\{\infty\}\big\}\cup\big\{\Lambda^r_{\rm loc}|\ r\in (0,\infty)\big\}$$
  a regularity class. A holomorphic principal $G$-bundle of class $\rg$ on $X$ is a principal 	${\cal O}^{\rg\,G}_X$-bundle on $X$ in the sense of Definition \ref{calG-bdl}.
 \end{dt}

 Now let $X$, $Y$ be complex manifolds with boundary in ${\cal X}$, respectively $Y$ and $f:X\to Y$ a differentiable map which restricts to a holomorphic map 
 $$f_0:\inte(X)\to\inte(Y).$$
  For any relatively open set $V\subset Y$ and map $u\in {\cal O}^{\rg\,G}_Y(V)$, the composition $u\circ f$ belongs to ${\cal O}^{\rg\,G}_X(f^{-1}(V))$. 

 Indeed, taking into account Proposition \ref{map-X-Y} and Remark \ref{restr-Y-X-M} we only have to show that the restriction of $u\circ f:f^{-1}(V)\to G$ to $f^{-1}(V)\cap\inte(X)$  is holomorphic. But  
 $$u\circ f|_{f^{-1}(V)\cap\inte(X)}=u|_{V\cap\inte(Y)}\circ f_0|_{f^{-1}(V)\cap\inte(X)},
 $$
 which is obviously holomorphic.
 
  This shows that
  \begin{pr}\label{pull-back-Lambda-hol}
Let $X$, $Y$ be complex manifolds with boundary in ${\cal X}$, respectively $Y$ and let $f:X\to Y$ a differentiable map which restricts to a holomorphic map $f_0:\inte(X)\to\inte(Y)$. 

Let $G$ be a complex Lie group and 
 $$\rg\in\big\{{\cal C}^k|\ k\in\N\cup\{\infty\}\big\}\cup\big\{\Lambda^r_{\rm loc}|\ r\in (0,\infty)\big\}$$
one of the regularity classes considered above.
  
The compatibility condition $C({\cal O}^{\rg\,G}_X,{\cal O}^{\rg\,G}_Y,f)$ introduced in Definition \ref{pull-back-bundle} is satisfied. In particular, for any holomorphic principal $G$-bundle ${\cal T}$ of  class  $\rg$ on $Y$, we have a well defined  pull-back holomorphic principal $G$-bundle  $f^*({\cal T})$  of  class  $\rg$  on $X$.
\end{pr}

Consider the special case when $X$, $Y$ are both complex manifolds with boundary in the same complex manifold ${\cal X}$ such that $X\subset Y$. 
\begin{re}\label{restrict-r-bdl-rem}
The condition $X\subset Y$ implies $\inte(X)\subset\inte(Y)$, so the inclusion map $\iota:X\hookrightarrow Y$ induces the inclusion $\iota_0:\inte(X)\hookrightarrow \inte(Y)$, which is obviously holomorphic. Therefore the hypothesis of Proposition \ref{pull-back-Lambda-hol} is fulfilled. 
\end{re}
This allows us to define:
\begin{dt}\label{restrict-r-bdl-def}
Let $X$, $Y$ be complex manifolds with boundary in the same complex manifold ${\cal X}$ such that $X\subset Y$, and let $\iota:X\hookrightarrow$ be the inclusion map. 
The pull-back bundle  $\iota^*({\cal T})$ of a holomorphic principal $G$-bundle ${\cal T}$ of  class  $\rg$ on $Y$   will be called the restriction of ${\cal T}$ to $X$.

\end{dt}

\subsection{Gluing torsors}

Let $X$ be a topological space,  $(U_i)_{i\in I}$  an open cover of $X$, and ${\cal G}$ be a  sheaf of  groups on $X$.  Denote by ${\cal G}_i$, ${\cal G}_{ij}$ the restriction of ${\cal G}$ to $U_i$, respectively $U_i\cap U_j$.

\begin{dt}
 A ${\cal G}$-torsor gluing data subordinate to 	$(U_i)_{i\in I}$ is a pair 
 $$\big(({\cal S}_i)_{i\in I},(\sigma_i^j)_{(i,j)\in I^2}\big)$$
  consisting of 
\begin{enumerate}
\item a family $({\cal S}_i)_{i\in I}$ of torsors, ${\cal S}_i$ being a  ${\cal G}_i$-torsor on $U_i$.
\item a family $(\sigma_i^j)_{(i,j)\in I^2}$ of torsor isomorphisms, $\sigma_i^j:{\cal S}_i|_{U_i\cap U_j}\textmap{\simeq}{\cal S}_j|_{U_i\cap U_j}$ being an isomorphism of ${\cal G}_{ij}$-torsors, satisfying the usual cocycle conditions:  
 \begin{enumerate}
	\item $\sigma_i^i=\id_{{\cal S}_i}$  for any $i\in I$.
	\item For each triple $(i,j,k)\in I^3$  we have $\sigma_i^k=\sigma_j^k \circ \sigma_i^j$ on $U_i\cap U_j \cap U_k$.
\end{enumerate}
	\end{enumerate}

\end{dt}

The following gluing principle  for torsors is an easy consequence of the well known gluing principle for sheaves of sets, see for instance  \cite[Lemma 6.33.2]{Stacks}:
\begin{pr}\label{GluingProp}
To any ${\cal G}$-torsor gluing data $\big(({\cal S}_i)_{i\in I},(\sigma_i^j)_{(i,j)\in I^2}\big)$ subordinate to $(U_i)_{i\in I}$ there is a canonically associated    pair $\big({\cal S},(\sigma_i: {\cal S}_i \textmap{\simeq} {\cal S}|_{U_i} )_{i\in I}\big)$, where ${\cal S}$ is a  ${\cal G}$-torsor on $X$ and for any $i\in I$, $\sigma_i: {\cal S}_i \textmap{\simeq} {\cal S}|_{U_i} $ is an isomorphism of  ${\cal G}_i$-torsors, which satisfies: 
\begin{enumerate}
\item The following compatibility conditions: 
\begin{equation}\label{CompCond}
\sigma_i=\sigma_j \circ \sigma_i^j \hbox{ on }U_i\cap U_j \hbox{ for any }i,\ j\in I.  	
\end{equation}

\item The following universal property: 

For any pair $({\cal T},(\theta_i: {\cal S}_i \to  {\cal T}|_{U_i} )_{i\in I})$ where ${\cal T}$  is a  ${\cal G}$-torsor  on $X$, and, for any $i\in I$, $\theta_i$ is an isomorphism of ${\cal G}_i$-torsors,    satisfying the compatibility conditions
\begin{equation}\label{CompCond-new}
\theta_i=\theta_j \circ \sigma_i^j \hbox{ on }U_i\cap U_j  \hbox{ for any }i,\ j\in I,
\end{equation}
there exists a  unique  ${\cal G}$-torsor isomorphism $\theta:{\cal S}\to  {\cal T}$ such that 
\begin{equation}\label{g_i-psi_i}
\theta|_{U_i}\circ \sigma_i=\theta_i \hbox{ for any }i\in I. 
\end{equation}
\end{enumerate}
\end{pr}

More precisely: for a fixed open cover 	$(U_i)_{i\in I}$ of $X$, we have a groupoid whose whose objects are the  ${\cal G}$-torsor gluing data subordinate to $(U_i)_{i\in I}$, and which is equivalent to the groupoid of  ${\cal G}$-torsors on $X$.  

A very simple special case of Proposition \ref{GluingProp} is the following 
\begin{co}\label{general-extension}
Let $(V_0,V_1)$  be an open cover of $X$,  ${\cal G}$  a  sheaf of  groups on $X$ and let ${\cal G}_0$, ${\cal G}_1$, ${\cal G}_{01}$  the restrictions of ${\cal G}$ to $V_0$, $V_1$ and $V_{01}\edf V_0\cap V_1$ respectively. Let ${\cal S}_i$ be  ${\cal G}_i$-torsors on $V_i$ for $0\leq i\leq 1$,  and let 
$\sigma: {\cal S}_0|_{V_{01}}\to {\cal S}_1|_{V_{01}}$ be an isomorphism of ${\cal G}_{01}$-torsors.  	
There exists a canonically associated triple $({\cal S}_0\coprod_\sigma{\cal S}_1,\sigma_0,\sigma_1)$ consisting of a ${\cal G}$-torsor ${\cal S}_0\coprod_\sigma{\cal S}_1$ on $X$ and isomorphisms 
$\sigma_i:{\cal S}_i\textmap{\simeq}({\cal S}_0\coprod_\sigma{\cal S}_1)|_{V_i}$ 
 satisfying $\sigma_0|_{V_{01}}=\sigma_1|_{V_{01}}\circ\sigma $  and  the following universal property:

For any ${\cal G}$-torsor ${\cal T}$ on $X$ and any  pair $(\theta_0,\theta_1)$ of torsor isomorphisms 
$$\theta_i:{\cal S}_i\textmap{\simeq} {\cal T}|_{V_i}$$
satisfying the compatibility condition $\theta_0|_{V_{01}}=\theta_1|_{V_{01}}\circ \sigma$, there is a unique ${\cal G}$-torsor isomorphism $\theta:{\cal S}_0\coprod_\sigma{\cal S}_1\textmap{\simeq} {\cal T}$ such that $\theta|_{V_i}\circ \sigma_i=\theta_i$ for $0\leq i\leq 1$.
\end{co}

In the next section we will apply this gluing principle in the special case when ${\cal S}_1$ is trivial. In this case the statement can be regarded as an extension principle (from $V_0$ to $X$) for a torsor ${\cal S}_0$ defined on $V_0$ which is trivial on $V_{01}$.
 
The pull back functor in the sense of Definition \ref{pull-back-bundle}, when  defined, commutes with gluing in the sense of Proposition \ref{GluingProp}. We state this property in the special case considered in Corollary \ref{general-extension}:

\begin{re}\label{restr-comm-with-gl} (pull-back commutes with gluing) Let $f:X\to X'$ be a continuous map between topological spaces, and let ${\cal G}\subset {\cal C}^G_X$, ${\cal G}'\subset{\cal C}^G_{X'}$ be subsheaves of subgroups on $X$, respectively $X'$ such that the compatibility condition $C({\cal G},{\cal G}',f)$ is satisfied. 

Let $(V'_0,V'_1)$  be an open cover of $X'$,  ${\cal S}'_i$ be  principal ${\cal G}_i$-bundles on $V'_i$,  and let
$$\sigma': {\cal S}'_0|_{V'_{01}}\to {\cal S}'_1|_{V'_{01}}$$
be an isomorphism.  Put $V_i\edf f^{-1}(V'_i)$ and let ${\cal S}_i\edf f^*({\cal S}'_i)$, $\sigma\edf f^*(\sigma')$ be the objects obtained by applying the pull back functor introduced in Definition \ref{pull-back-bundle} to ${\cal S}'_i$, $\sigma'$ respectively. Then we have a canonical isomorphism  
$$f^*\big({\cal S}'_0\coprod_{\sigma'}{\cal S}'_1\big)\simeq {\cal S}_0\coprod_{\sigma}{\cal S}_1.$$

\end{re}

\begin{proof}
It suffices to note that the inverse image functor $f^{-1}$ and 	the group sheaf change functor intervening in the definition of the pull-back commutes with gluing.  This follows easily using the universal property of the torsor associated with gluing data taking into account that both functors  commute with restriction to open sets.
\end{proof}

%

%
%
%
%
%
%

\section{The proof of the main theorem}

An important ingredient in the proof of Theorem \ref{mainTh} is the Oka principle for bundles on strictly pseudoconvex compact manifolds with boundary in Stein manifolds  (see \cite[Theorem 1.7]{DrFo}, \cite[Theorem 5.4.11]{Fo-book}, \cite[Theorem 7.1]{Fo}):

\begin{pr}\label{Oka} [Oka principle for compact manifolds with boundary]
Let ${\cal X}$ be a Stein manifold, $D\Subset{\cal X}$  a  smoothly bounded strictly pseudoconvex domain in ${\cal X}$, $G$ a complex Lie Group, and
 $$\rg\in\big\{{\cal C}^k|\ k\in\N\cup\{\infty\}\big\}\cup\big\{\Lambda^r_{\rm loc}|\ r\in (0,\infty)\big\}$$
 one of the regularity classes considered above.
 
 Let ${\cal P}$ be a holomorphic principal $G$-bundle  of class $\rg$ on $\bar D$. Then any continuous section of (the underlying topological bundle of) ${\cal P}$ is homotopic to a section of ${\cal P}$.
 
 In particular ${\cal P}$ is trivial if and only if its underlying topological bundle is trivial.  
\end{pr}

\begin{proof}

For $\rg\in \big\{\Lambda^r_{\rm loc}|\ r\in (0,\infty)\big\}$, the claim follows directly from \cite[Theorem 1.1]{Fo}, whereas for $\rg\in {\cal C}^k|\ k\in\N$,  it is a special case of \cite[Theorem 1.7]{DrFo}.

Suppose now that   $\rg= {\cal C}^\infty$.  Let $P$ be the underlying differentiable bundle of ${\cal P}$ and $s$  a continuous section of $P$. Choosing   an open neighbourhood $U$ of $\bar D$ in which $\bar D$ is a differentiable deformation retract, we obtain a  principal $G$-bundle $\Pi$ on $U$  extending $P$ and a continuous extension $\sigma$ of $s$. On the other hand, by    \cite[Theorem 1.1]{Te}, it follows  that ${\cal P}$ extends to a holomorphic principal $G$-bundle $ {\cal P}'$ on an open  $U'$ of $\bar D$ in $U$. We may assume that $U'$ is still strictly pseudoconvex, so Stein. By the Grauert's Oka principle on open Stein manifolds \cite{GrKe}, \cite{Ra}, it follows that $\sigma'\edf \sigma|_{U'}$ is homotopic to a holomorphic section $h'\in \Gamma(U',{\cal P}')$. Therefore, restricting to $X$, we obtain holomorphic section $h\in \Gamma(X,{\cal P})$ homotopic to $s=\sigma'|_X$. 

\end{proof} 
Using Proposition \ref{Oka}, we will prove two extension results for bumps needed in the proof of Theorem \ref{mainTh}:
 \begin{figure}[h]
\includegraphics[scale=0.7]{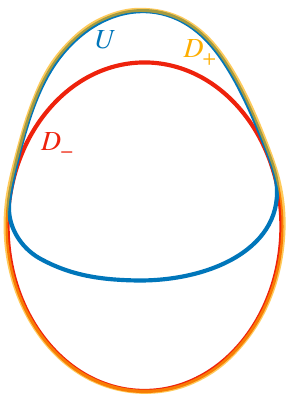}	
\caption{A convex bump}
 \label{bump}
\end{figure}

\begin{pr}\label{ExtBdlOnBump1}
\begin{enumerate}
\item Let $[D_-,U,D_+]$ be a convex bump in ${\cal X}$ in the sense of \cite[Definition 2.1]{HL} and ${\cal P}$ a holomorphic principal $G$-bundle  of class $\rg$ on $\bar D_-$. Then ${\cal P}$ 
   is the restriction  of a holomorphic principal $G$-bundle ${\cal Q}$ of class $\rg$ on $\bar D_+$ which is trivial on $\bar U$. 
\item Let $[D_-,U,D_+]$ be a special pseudoconvex bump in ${\cal X}$ in the sense of \cite[Definition 2.6]{HL} and ${\cal P}$ a holomorphic principal $G$-bundle of class $\rg$ on $\bar D_-$. Let $P'$ be a topological $G$-bundle on $\bar D_+$ whose restriction to $\bar D_-$ is isomorphic to
 the underlying topological bundle $P$ of ${\cal P}$. Then ${\cal P}$ is the restriction  of a holomorphic principal $G$-bundle ${\cal Q}$ of class $\rg$ on $\bar D_+$ which is trivial on $\bar U$ and whose underlying topological bundle is isomorphic to $P'$.  
\end{enumerate}

%

%
\end{pr}

\begin{proof}
	
(1)  We know that $[D_-,U,D_+]$ is a convex bump in ${\cal X}$ (see fig. \ref{bump}), in particular $D_-$, $D_+$, $U$, $D_-\cap U$ are all strictly pseudoconvex domains in ${\cal X}$, which implies in particular that their closures   $\bar D_-$, $\bar D_+$, $\bar U$, $\overline{D_-\cap U}$ are all  compact complex manifolds with boundary in ${\cal X}$ in the sense of  Definition \ref{XcomplexincalX}.

We will use Corollary \ref{general-extension} by taking $X=\bar D^+$ and
 \begin{equation}
 \begin{split}
 V_0\edf\bar D_+\;\setminus\; \overline{U\;\setminus\; D_-}&=\inte_{\bar D_+}(\bar D_-),\\ 
 V_1\edf \bar D_+\;\setminus\; \overline{D_-\;\setminus\; U}&=\inte_{\bar D_+}(\bar U).	
 \end{split}	
 \end{equation}
 In these formulae and in the rest of the proof we use overline for closure in the ambient manifold ${\cal X}$.
 The second equalities on the right are the formulae (\ref{3g}), (\ref{3g'}) in Lemma \ref{generalized-bump-lemma} proved below. Note that $V_\pm$ are obviously relatively open subsets of $\bar D^+$ and they cover $\bar D^+$. Indeed, we have
\begin{equation}\label{V-cupV+}
\begin{split}
V_0\cup V_1=&(\bar D_+\;\setminus\; \overline{D_-\;\setminus\; U})\cup (\bar D_+\;\setminus\; \overline{U\;\setminus\; D_-})=\bar D_+\;\setminus\; \big(\overline{D_-\;\setminus\; U}\cap \overline{U\;\setminus\; D_-})\\
=&\bar D_+\;\setminus\; \emptyset =\bar D_+.	
\end{split}
\end{equation}
For the second equality we have used $\overline{D_-\;\setminus\; U}\cap \overline{U\;\setminus\; D_-}=\emptyset$, which is a property of a convex bump (see [Definition 2.1 (iii) p. 76] \cite{HL}).

Using the equalities (\ref{3g}), (\ref{3g'}), (\ref{4g}) given by Lemma \ref{generalized-bump-lemma} proved below, we have
\begin{equation}
\begin{split}
V_0\cap V_1&= \big(\bar D_+\;\setminus\; \overline{U\;\setminus\; D_-}\big)\cap \big(\bar D_+\;\setminus\; \overline{D_-\;\setminus\; U}\big)=\inte_{\bar D_+}(\bar D_-)\cap \inte_{\bar D_+}(\bar U) \\
&=\inte_{\bar D_+}(\bar D_-\cap \bar U)\stackrel{(\ref{4g})}{=}\inte_{\bar D_+}\overline{D_-\cap U}.	
\end{split}	
\end{equation}
%
%
We apply the restriction functor (see (see  Remark \ref{restrict-r-bdl-rem}, Definition \ref{restrict-r-bdl-def}) to the following two inclusions of complex manifolds with boundary in ${\cal X}$: 
$$V_0\cap V_1=\inte_{\bar D_+}\overline{D_-\cap U}\subset \overline{D_-\cap U}\subset \bar D_-.
$$

The restriction ${\cal P}|_{V_0\cap V_1}$ can be obtained as the restriction to $V_0\cap V_1$ of ${\cal P}|_{\overline{D_-\cap U}}$. 
%
%
We will now make use of the Oka principle (Proposition \ref{Oka})  to the latter restriction. By the definition of a convex bump, the image of $\overline{D_-\cap U}$ by a suitable holomorphic chart $h$ defined on a neighbourhood of $\bar U$ in ${\cal X}$ is strictly convex (in particular contractible and strictly pseudoconvex) in the Stein manifold $\C^n$. Recalling that any topological principal bundle on a contractible space is  trivial, it follows that the underlying topological bundle of ${\cal P}|_{\overline{D_-\cap U}}$  is trivial, so, by Proposition \ref{Oka} (the Oka principle), ${\cal P}|_{\overline{D_-\cap U}}$  is trivial.
Let $\sigma$ be a torsor isomorphism   
$${\cal P}|_{\overline{D_-\cap U}}\,\textmap{\simeq}\,{\cal O}^{r\,G}_{\overline{D_-\cap U}}\,,$$
 and let $\sigma_{01}$, $\sigma'_1$, $\sigma''_0$ be its restrictions to the relatively open subsets 
 $$
\begin{array}{ccccc}
V_0\cap V_1&=&\inte_{\bar D_+}(\bar D_-\cap \bar U)&\stackrel{(\ref{4g})}{=}&\inte_{\bar D_+}\overline{D_-\cap U}  ,\\ 
V'_1&\edf& \inte_{\bar D_-}(\bar D_-\cap   \bar U)&\stackrel{(\ref{4g})}{=}&\inte_{\bar D_-}\overline{D_-\cap   U}  ,\\ 
V''_0&\edf& \inte_{\bar U}(\bar D_-\cap\bar U)&\stackrel{(\ref{4g})}{=}&\inte_{\bar U}\overline{ D_-\cap\bar U} 	
\end{array}
$$
of $\overline{D_-\cap U}	$ respectively. Using the torsors ${\cal P}|_{V_0}$, ${\cal O}^{r\,G}_{V_1}$ and the isomorphism $\sigma_{01}$ we obtain by Corollary \ref{general-extension}  a holomorphic principal $G$-bundle  of class $\rg$
$$
{\cal Q}\edf {\cal P}|_{V_0}\coprod_{\sigma_{01}}{\cal O}^{r\,G}_{V_1}
$$
 on $\bar D_+$ whose restriction to $V_0=\inte_{\bar D_+}(\bar D_-)$ is canonically isomorphic to ${\cal P}|_{V_0}$ and which is trivial on $V_1=\inte_{\bar D_+}(\bar U)$. This does not complete the proof yet, because we need  isomorphisms 
$${\cal P}\simeq {\cal Q}|_{\bar D_-},\ {\cal O}^{r\,G}_{\bar U}\simeq{\cal Q}|_{\bar U},$$
on the whole closed subspaces $\bar D_-$, $\bar U$, whereas the construction of ${\cal Q}$ gives only   isomorphisms 
 $$\sigma_0:{\cal P}|_{\inte_{\bar D_+}(\bar D_-)}\textmap{\simeq}{\cal Q}|_{\inte_{\bar D_+}(\bar D_-)} ,\ \sigma_1:{\cal O}^{r\,G}_{\inte_{\bar D_+}(\bar U)}\textmap{\simeq}{\cal Q}|_{\inte_{\bar D_+}(\bar U)}$$
 satisfying the compatibility condition $\sigma_0|_{V_0\cap V_1}=\sigma_1|_{V_0\cap V_1}\circ\sigma_{01}$.
  
Consider the open cover   $(V_0\cap \bar D_-,V_1\cap \bar D_-)$ of $\bar D_-$ induced by $(V_0,V_1)$. We have $V_0\cap \bar D_-=V_0$ and 
\begin{align*}
V_1\cap \bar D_-=&\bar D_-\;\setminus\; \overline{D_-\;\setminus\; U}\overset{(\ref{2g})}{=}\bar D_-\;\setminus\; \overline{\bar D_-\;\setminus\; \bar U}=\complement_{\bar D_-}\overline{\bar D_-\Setminus \bar U}=\inte_{\bar D_-}\big(\complement_{\bar D_-}(\bar D_-\Setminus \bar U)\big)=\\
=&\inte_{\bar D_-}\big(\complement_{\bar D_-}(\bar D_-\Setminus (\bar D_-\cap \bar U))\big)=\inte_{\bar D_-}(\bar D_-\cap \bar U)=V'_1.
\end{align*}

Since restriction commutes with gluing (Remark \ref{restr-comm-with-gl}), it follows that the restriction ${\cal Q}'\edf {\cal Q}|_{\bar D_-}$ can be obtained by gluing, more precisely by applying  Corollary \ref{general-extension} to the open cover $(V_0\cap \bar D_-,V_1\cap \bar D_-)=(V_0,V'_1)$ of $\bar D_-$ and to the restrictions 
$$({\cal P}|_{V_0})|_{V_0\cap\bar D_-}={\cal P}|_{V_0},\ {\cal O}^{r\,G}_{V_1}|_{V'_1}={\cal O}^{r\,G}_{V'_1},\ \sigma_{01}'\edf \sigma_{01}|_{V_0\cap V_1\cap \bar D_-}=\sigma_{01}$$
 of ${\cal P}|_{V_0}$, ${\cal O}^{r\,G}_{V_1}$ and $\sigma_{01}$ to $V_0\cap \bar D_-=V_0$, $V_1\cap \bar D_-=V'_1$,  and $V_0\cap V_1\cap\bar D_-=V_0\cap V_1$ respectively. In other words we have a canonical identification
 $$
 {\cal Q}'={\cal P}|_{V_0}\coprod_{\sigma_{01}}{\cal O}^{r\,G}_{V'_1}.
 $$

We define an isomorphism $\theta':{\cal Q}'\textmap{\simeq}{\cal P}$ by applying the universal property of the torsor associated with the gluing data $({\cal P}|_{V_0}\,,\,{\cal O}^{r\,G}_{V'_1}\,,\,\sigma_{01})$ (see the second part of the conclusion of Corollary \ref{general-extension}), and choosing 
$$\theta'_0=\id_{{\cal P}|_{V_0}},\ \theta'_1=\sigma'^{-1}_1.$$
The required compatibility condition $\theta'_0|_{V'_0\cap V'_1}=\theta'_1|_{V'_0\cap V'_1}\circ \sigma_{01}$  is satisfied, because, by definition, the isomomorphism $\sigma'_1$ coincides with $\sigma_{01}$ on $V_0\cap V_1$.
This shows that ${\cal Q}'\edf {\cal Q}|_{\bar D_-}$ is isomorphic to ${\cal P}$ as claimed.

Similarly, consider  the open cover $(V_0\cap \bar U,V_1\cap \bar U)$  of $\bar U$ induced by $(V_0,V_1)$. We have  $V_1\cap \bar U=V_1$ and
\begin{align*}
V_0\cap \bar U=&\bar U\;\setminus\; \overline{U\;\setminus\; D_-}\overset{(\ref{1g})}{=}\bar U\;\setminus\; \overline{\bar U\;\setminus\; \bar D_-}=\complement_{\bar U}\,\overline{\bar U\;\setminus\; \bar D_-}=\inte_{\bar U}(\complement_{\bar U}(\bar U\;\setminus\; \bar D_-))=\\
=&\inte_{\bar U}(\complement_{\bar U}(\bar U\;\setminus\; (\bar D_-\cap\bar U)))=\inte_{\bar U}(\bar D_-\cap\bar U)=V''_0.
\end{align*}

The restriction ${\cal Q}''\edf {\cal Q}|_{\bar U}$ can be obtained by applying  Corollary \ref{general-extension} to the open cover $(V_0\cap\bar U,V_1\cap\bar U)$ of $\bar U$ and to the restrictions 
$$({\cal P}|_{V_0})|_{V_0\cap\bar U}={\cal P}|_{V''_0},\ {\cal O}^{r\,G}_{V_1}|_{V_1\cap\bar U}={\cal O}^{r\,G}_{V_1},\ \sigma_{01}''\edf \sigma_{01}|_{V_0\cap V_1\cap\bar U}=\sigma_{01}$$
 of ${\cal P}|_{V_0}$, ${\cal O}^{r\,G}_{V_1}$ and $\sigma_{01}$ to $V_0\cap\bar U=V''_0$, $V_1\cap\bar U=V_1$,  and $V_0\cap V_1\cap\bar U=V_0\cap V_1$ respectively. Therefore
 $$
 {\cal Q}''={\cal P}|_{V''_0}\coprod_{\sigma_{01}}{\cal O}^{r\,G}_{V_1}.
 $$

We define an isomorphism $\theta'':{\cal Q}''\textmap{\simeq}{\cal O}^{r\,G}_{\bar U}$ using   the universal property of the torsor associated with the gluing data $({\cal P}|_{V''_0}\,,\,{\cal O}^{r\,G}_{V_1}\,,\,\sigma_{01})$  and choosing this time
$$\theta''_0=\sigma''_0,\ \theta''_1=\id_{{\cal O}^{r\,G}_{V_1}}.$$
The required compatibility condition $\theta''_0|_{V_0\cap V_1}=\theta''_1|_{V_0\cap V_1}\circ \sigma_{01}$  is satisfied, because  $\sigma''_0$ coincides with $\sigma_{01}$ on $V'_0\cap V'_1=V_0\cap V_1$.
This shows that ${\cal Q}''\edf {\cal Q}|_{\bar U}$ is isomorphic to ${\cal O}^{r\,G}_{\bar U}$ as claimed.
\vspace{2mm}\\
(2) By definition of a special pseudoconvex bump, we know that $U$ is starshaped, in particular contractible. Since a compact manifold with boundary is homotopically equivalent to its interior, it follows that $\bar U$ is contractible too, so $P'|_{\bar U}$ is trivial. Let 
$$t'_{\bar U}:P'|_{\bar U}\to \bar U\times G$$ 
a trivialisation of $P'|_{\bar U}$ and 
$$t'_{\overline{D_-\cap U}}:P'|_{\overline{D_-\cap U}} \textmap{\simeq} \overline{D_-\cap U}\times G$$
 its restriction to $\overline{D_-\cap U}$. Fix a bundle isomorphism   $u: P\textmap{\simeq} P'|_{\bar D_-}$ and let 
\begin{equation}\label{t'->t}
t_{\overline{D_-\cap U}}\edf  t'_{\overline{D_-\cap U}}\circ u|_{\overline{D_-\cap U}}:P|_{\overline{D_-\cap U}}\textmap{\simeq}\overline{D_-\cap U}\times G 	
\end{equation}
 the corresponding trivialisation of $P|_{\overline{D_-\cap U}}$. By the definition of a special pesudoconvex bump, we know that, via a suitable holomorphic chart defined in an open neighbourhood of $\bar U$, $D_-\cap U$ can be identified with (in general not necessarily connected) special strictly pseudoconvex domain in $\C^n$. Therefore  the Oka principle applies, and yields a trivialisation
 $$\sigma=\sigma_{\overline{D_-\cap U}}:{\cal P}|_{\overline{D_-\cap U}}\textmap{\simeq}{\cal O}^{r\,G}_{\overline{D_-\cap U}}$$
 of the torsor ${\cal P}|_{\overline{D_-\cap U}}$ whose associated topological trivialisation
 $$s=s_{\overline{D_-\cap U}}:P|_{\overline{D_-\cap U}} \textmap{\simeq}\overline{D_-\cap U}\times G $$
  is homotopic with $t\edf t_{\overline{D_-\cap U}}$. Here we have used the standard identification between trivialisations and sections.
  
  Denoting by   $\sigma_{01}$ the restriction of $\sigma$ to $V_0\cap V_1$, we obtain  as in the first part of the lemma an extension 
  $${\cal Q}\edf {\cal P}|_{V_0}\coprod_{\sigma_{01}}{\cal O}^{r\,G}_{V_1}$$
   of ${\cal P}$ to $\bar D_+$ which is trivial on $\bar U$. The underlying topological bundle $Q$ of ${\cal Q}$ is
 $$
 Q=P|_{V_0}\coprod_{s_{01}}(V_1\times G),
 $$
 where $s_{01}$ stands for the restriction of $s$ to $V_0\cap V_1$. Since $s$ is homotopic to $t$, it follows that $s_{01}$ is homotopic to the restriction $t_{01}$ of $t$ to $V_0\cap V_1$, so
 $$Q\simeq P|_{V_0}\coprod_{t_{01}}(V_1\times G).$$
 We define an isomorphism $P|_{V_0}\coprod_{t_{01}}(V_1\times G)\textmap{\simeq} P'$  using again the universal property of the sheaf associated with gluing data, and taking
 $$
\theta_0: P|_{V_0}\textmap{\simeq} P'|_{V_0},\ \theta_1:V_1\times G\textmap{\simeq} P'_{V_1}
 $$
 defined by
  $$
 \theta_0\edf u|_{V_0},\ \theta_1=t'|_{V_1}^{-1},
 $$
 where $t'\edf t'_{\overline{D_-\cap U}}$.
The compatibility condition  $\theta_0|_{V_0\cap V_1}=\theta_1|_{V_0\cap V_1}\circ t_{01}$ is satisfied by the definition of $t$ in terms of $t'$ (formula (\ref{t'->t})).
 %
\end{proof}
More generally,

\begin{pr} \label{ExtBdlOnBump2}
\begin{enumerate}
\item Let $[D_-,U,D_+]$ be a convex bump in ${\cal X}$, $K\subset D_-\setminus \bar U$ a compact set,   and ${\cal P}$ a holomorphic principal $G$-bundle of class $\rg$ on $\bar D_-\setminus K$. Then ${\cal P}$ 
   is the restriction  of a holomorphic principal $G$-bundle ${\cal Q}$ of class $\rg$ on $\bar D_+\setminus K$ which is trivial on $\bar U$. 
\item Let $[D_-,U,D_+]$ be a special pseudoconvex bump in ${\cal X}$, $K\subset D_-\setminus \bar U$ a compact set and ${\cal P}$ a holomorphic principal $G$-bundle of class $\rg$ on $\bar D_-\setminus K$. Let $P'$ be a topological $G$-bundle on $\bar D_+\setminus K$ whose restriction to $\bar D_-\setminus K$ is isomorphic to
 the underlying topological bundle $P$ of ${\cal P}$. Then ${\cal P}$ is the restriction  of a holomorphic principal $G$-bundle ${\cal Q}$ of class $\rg$ on $\bar D_+\setminus K$ which is trivial on $\bar U$ and whose underlying topological bundle is isomorphic to $P'$.  
\end{enumerate}

\end{pr}

\begin{proof}
  In the proof of Proposition \ref{ExtBdlOnBump2} we replace everywhere ${\cal X}$ by ${\cal X}^K\edf{\cal X}\setminus K$, $D_\pm$ by  $D_\pm^K\edf D_\pm\setminus K$ and closures with respect to ${\cal X}$ by closures with respect to ${\cal X}^K$.

 The closure of $U$ with respect to ${\cal X}^K$ coincides with  its closure with respect to ${\cal X}$, so it remains compact; it is a closed manifold with boundary in both ${\cal X}$ and ${\cal X}^K$. Using overline for closures with respect to ${\cal X}^K$, we see that $\bar D_\pm^K$ are (in general non-compact) but still closed manifolds with boundary in ${\cal X}^K$. Moreover we still have   $D_+^K=D_-^K\cup U$. Therefore Lemma \ref{generalized-bump-lemma} applies to ${\cal X}^K$ and its open subsets $D_\pm^K$, $U$. We still have 
 $$
\overline{D_-^K\setminus U}\cap \overline{U \setminus D_-^K}=\emptyset,
$$
because $D_-^K\setminus U\subset D_-\setminus U$, $U \setminus D_-^K=U\setminus D_-$, and closures with respect to ${\cal X}^K$ are obtained by intersecting with ${\cal X}^K$ closures with respect to ${\cal X}$. This gives the analogue of formula (\ref{V-cupV+}) for the new manifolds ${\cal X}^K$, $D_\pm^K$, $U$ and the relatively open sets $V_\pm^K\subset \bar D_+^K$ defined similarly.

  Moreover $D_-^K\cap U=D^-\cap U$ and its closure in ${\cal X}$ is contained in $\bar U$ so in ${\cal X}^K$, so it coincides with its closure $\overline{D_-^K\cap U}$ in ${\cal X}^K$. Therefore the same arguments based on the Oka principle (Proposition \ref{Oka}) apply. With these remarks, the proof of Proposition \ref{ExtBdlOnBump1} applies  verbatim to Proposition \ref{ExtBdlOnBump2}.



 \end{proof}

\begin{lm}\label{generalized-bump-lemma} Let ${\cal X}$ be a differentiable manifold, and let $D_-$, $U\subset {\cal X}$ be open subsets such that $\bar D_-$, $\bar U$ are manifolds with boundary in ${\cal X}$ in the sense of Definition \ref{XincalX}. Put $D_+\edf D_-\cup U$.
 
\begin{enumerate}
\item 	The following equalities hold:
\begin{align}
\overline{U\;\setminus\;D_-}=\overline{U\;\setminus\; \bar D_-}&=\overline{\bar U\;\setminus\; \bar D_-}\,. \label{1g}\\
\overline{D_-\;\setminus\; U}=\overline{D_-\;\setminus\; \bar U}&=\overline{\bar D_-\;\setminus\; \bar U}\,.\label{2g}
\end{align}
\begin{align}
\bar D_+\;\setminus\; \overline{U\;\setminus\; D_-}&=\inte_{\bar D_+}(\bar D_-)\,.\label{3g}\\
\bar D_+\;\setminus\; \overline{D_-\;\setminus\; U}&=\inte_{\bar D_+}(\bar U)\,.\label{3g'}	
\end{align}
\item  If $\bar D_+$ is also a manifold with boundary in ${\cal X}$, then
\begin{equation}
\bar U\cap\bar D_-=\overline{U\cap D_-}\,. \label{4g}	
\end{equation}
\end{enumerate}
\end{lm}

\begin{proof}
(1) For the first equality in (\ref{1g}):   We have
$$
U\;\setminus\; D_-=U\cap \complement_{\cal X}(D_-)=U\cap \big(\complement_{\cal X}(\bar D_-)\cup\partial\bar D_-)=(U\;\setminus\;\bar D_-)\cup (U\cap\partial\bar D_-). 
$$
Therefore, it suffices to prove that any point of  $U\cap\partial\bar D_-$ is adherent to $U\setminus\bar D_-$. Since $\bar D_-$ is a closed manifold with boundary in ${\cal X}$, by Remark \ref{subman-with-bd} (\ref{2nd}) we know that $\complement_{\cal X}(D_-)$ is a manifold with boundary in ${\cal X}$ whose boundary is
$$
\partial\big(\complement_{\cal X}(D_-)\big)=\partial(\bar D_-).
$$
and whose interior is
$$
\inte\big(\complement_{\cal X}(D_-)\big)=\complement_{\cal X}(\bar D_-).
$$
By Remark \ref{subman-with-bd} (\ref{1st}), it follows that $\complement_{\cal X}(D_-)\cap U$ is a manifold with boundary in $U$ whose boundary is $\partial(\bar D_-)\cap U$ and whose interior is $\complement_{\cal X}(\bar D_-)\cap U=U\Setminus\bar D_-$. Since any boundary point of a manifold with boundary is adherent to its interior, the claim follows. 

For the second equality in (\ref{1g}): We have
$$
\bar U\Setminus\bar D_-=\bar U\cap\complement_{\cal X}(\bar D_-)=\big(U\cap\complement_{\cal X}(\bar D_-)\big)\cup \big(\partial\bar U\cap\complement_{\cal X}(\bar D_-)\big)=(U\Setminus\bar D_-)\cup\big(\partial\bar U\Setminus \bar D_-\big) .
$$
Therefore, it suffices to prove that any point of $\partial\bar U\Setminus \bar D_-=\partial\bar U\cap\complement_{\cal X}(\bar D_-)$ is adherent to $U\Setminus\bar D_-=U\cap\complement_{\cal X}(\bar D_-)$. But, by Remark \ref{subman-with-bd} (\ref{1st}), $\partial\bar U\cap\complement_{\cal X}(\bar D_-)$ is the boundary of the manifold with boundary $\bar U\cap\complement_{\cal X}(\bar D_-)$ whose interior is $U\cap\complement_{\cal X}(\bar D_-)$, so the claim follows as in the proof of the first equality.
\\ \\
For (\ref{2g}): it suffices to interchange to roles of $D_-$ and $U$ in the proof of  (\ref{1g}).
\\ \\
For (\ref{3g}), (\ref{3g'}): Since  $D_+=D_-\cup U$, we have
$$\bar D_+=\bar D_-\cup\bar U=(\bar U\;\setminus\;\bar D_-)\cup \bar D_-=(\bar D_-\;\setminus\;\bar U)\cup \bar U.
$$
This shows that
\begin{equation}\label{complements}
\complement_{\bar D^+}(\bar U\;\setminus\;\bar D_-)=\bar D_-,\ \complement_{\bar D^+}(\bar D_-\;\setminus\;\bar U)=\bar U,
\end{equation}
hence, taking into account (\ref{1g}),
$$
\bar D_+\;\setminus\; \overline{U\;\setminus\; D_-}=\bar D_+\;\setminus\; \overline{\bar U\;\setminus\; \bar D_-}=\complement_{\bar D_+}\overline{\bar U\Setminus \bar D_-}=\inte_{\bar D_+}\big(\complement_{\bar D_+}(\bar U\Setminus \bar D_-)\big)=\inte_{\bar D_+}(\bar D_-),
$$
where, for the last equality, we have used (\ref{complements}).
Similarly, taking into account (\ref{2g})  and (\ref{complements}),
$$
\bar D_+\;\setminus\; \overline{D_-\;\setminus\; U}=\bar D_+\;\setminus\; \overline{\bar D_-\;\setminus\; \bar U}=\complement_{\bar D_+}\overline{\bar D_-\Setminus \bar U}=\inte_{\bar D_+}\big(\complement_{\bar D_+}(\bar D_-\Setminus \bar U)\big)=\inte_{\bar D_+}(\bar U).
$$
(2) The inclusion $\overline{U\cap D_-}\subset \bar U\cap   \bar D_-$ is clear. For the opposite inclusion: 
$$
\bar U\cap   \bar D_-=(U\cup\partial\bar U)\cap(D_-\cup\partial\bar D_-)=(U\cap D_-)\cup (U\cap\partial\bar D_-)\cup(\partial\bar U\cap D_-)\cup(\partial\bar U\cap\partial\bar D_-). 
$$
By  Remark \ref{subman-with-bd}, $U\cap\partial\bar D_-$ is the boundary of the manifold with boundary $U\cap\bar D_-$ whose interior is $U\cap D_-$, and $\partial\bar U\cap D_-$ is the boundary of the manifold with boundary $\bar U\cap D_-$ whose interior is $U\cap D_-$. Therefore, any point of $U\cap\partial\bar D_-$ and any point of  $\partial\bar U\cap D_-$ is adherent to $U\cap D_-$. So we only have to check that any point   $x\in \partial\bar U\cap\partial\bar D_-$ is also  adherent to   $U\cap D_-$.

We have $D_-\cup U=D_+$, so $\bar D_-\cup \bar U=\overline{D_-\cup U}=\bar D_+$, so $\bar D_-\subset\bar D_+$, $\bar U\subset\bar D_+$. Therefore, recalling that in this part of the lemma we have assumed that $\bar D_+$ is also a manifold with boundary in ${\cal X}$, we obtain
\begin{equation*}
\begin{split}
\partial\bar U\cap\partial\bar D_-=(\bar U\Setminus U)\cap (\bar D_-\Setminus D_-)\subset &\,(\bar D_+\Setminus U)\cap (\bar D_+\Setminus D_-)=\bar D_+\Setminus (U\cup D_-)=\\
&\,\bar D_+\setminus D_+=\partial\bar D_+.	
\end{split}
\end{equation*}
Therefore any point $x_0\in \partial\bar U\cap\partial\bar D_-$ belongs to $\partial\bar D_+$. Let $\gamma:(-1,1)\to {\cal X}$ be a smooth embedding such that 
\begin{itemize}
\item $\gamma(0)=x_0$ and $\dot\gamma(0)$ is transversal to the hypersurface $\partial\bar D_+$ at $x_0$,
\item  $\im(\gamma)\cap\partial\bar D_+=\{x_0\}$.
\item $\gamma((-1,0))\subset \complement_{\cal X}(\bar D_+)$	, $\gamma((0,1))\subset D_+$.
\end{itemize}

Since $\gamma((-1,0))\subset \complement_{\cal X}(\bar D_+)=\complement_{\cal X}(\bar D_-\cup\bar U)=\complement_{\cal X}(\bar D_-)\cap \complement_{\cal X}(\bar U)$, it follows   that the curve $\gamma((-1,0))$ is exterior to both manifolds with boundary $\bar D_-$, $\bar U$. Therefore, for sufficiently small $\varepsilon>0$, $\gamma((0,\varepsilon))$ must be interior to these manifolds with boundary so, for sufficiently large $n\in\N^*$, $\gamma\big(\frac{1}{n}\big)\in U\cap D_-$. On the other hand $\lim_{n\to\infty} \gamma\big(\frac{1}{n}\big)=x_0$, so $x_0$ is adherent to $ U\cap D_-$, as claimed.

\end{proof}

\begin{proof} (of Theorem \ref{mainTh}): 
(1) Assume that $X_0\Subset Z_0$ is a non-critical strictly pseudoconvex extension in ${\cal X}$. 

Let $V'\Subset V''\Subset X_0$ be open neighbourhoods with compact closures of $K$ in $X_0$ and note that $({\cal X}\setminus \bar V',V'')$ is an open cover of ${\cal X}$.

By \cite[Lemma 2.2 p. 76]{HL} applied to the non-critical strictly pseudoconvex extension $(X_0,Z_0)$ and the open cover $({\cal X}\setminus \bar V',V'')$, we know that there exists strictly pseudoconvex domains $D_j$ ($0\leq j\leq k+1$), $U_j$, ($0\leq j\leq k$) in ${\cal X}$ such that
\begin{itemize}
\item $X_0 = D_0$ and $Z_0=D_{k+1}$,	
\item  For any $0\leq j\leq k$, the triple $[D_j,U_j,D_{j+1}]$ is a convex bump in ${\cal X}$ and $U_j$ is contained in either ${\cal X}\setminus \bar V'$ or $V''$. 
\end{itemize}

The sequence of convex bumps $([D_j,U_j,D_{j+1}])_{0\leq j\leq k}$ yields inclusion
\begin{equation}\label{inclusions}
X_0=D_0\subset D_1\subset \dots 	\subset D_{k+1}=Z_0.
\end{equation}
We can of course assume all inclusions in (\ref{inclusions}) are strict, because, if, for a certain $j_0$ we have $D_{j_0}=D_{j_0+1}$, we can remove the convex bump $[D_{j_0},U_{j_0},D_{j_0+1}]$ from the sequence. Since $D_{j+1}=D_j\cup U_j$, it follows that for any $0\leq j\leq k$, $U_j\not\subset D_j$, so  $U_j\not\subset D_0=X_0$, so $U_j$ cannot be contained in $V''$. Therefore, for any $0\leq j\leq k$ we have $U_j\subset {\cal X}\setminus \bar V'$. This implies $U_j\subset {\cal X}\setminus   V'$, so $\bar U_j\subset {\cal X}\setminus   V'\subset {\cal X}\setminus K$.

Starting with ${\cal P}$ and applying inductively Proposition \ref{ExtBdlOnBump2} (1), we obtain holomorphic bundles ${\cal P}_j$ of class $\rg$ on $\bar D_j\setminus K$ such that ${\cal P}_0={\cal P}$ and ${\cal P}_{j+1}|_{D_j}\simeq{\cal P}_j$.  In particular we obtain a holomorphic bundle ${\cal Q}\edf {\cal P}_{k+1}$ of class $\rg$ on $\bar D_{k+1}\setminus K=Z\setminus K$ whose restriction to $\bar D_{0}\setminus K=X\setminus K$ is   ${\cal P}$.
\vspace{2mm}\\
(2) Assume now that  $X_0\Subset Z_0$ is a general strictly pseudoconvex extension in ${\cal X}$, and that the underlying topological bundle $P$ of   ${\cal P}$  is isomorphic to the restriction to $X\setminus K$ of a topological $G$-bundle $P'$ defined on $\bar Z\setminus K$. 

This time we make use \cite[Corollary 2.8]{HL} to the extension $(X_0,Z_0)$ and the open cover $({\cal X}\setminus \bar V',V'')$ of ${\cal X}$ defined as  above. We obtain strictly pseudoconvex domains $D_j$ ($0\leq j\leq k+1$), $U_j$, ($0\leq j\leq k$) in ${\cal X}$ such that $X_0 = D_0$,  $Z_0=D_{k+1}$, and  for any $0\leq j\leq k$, the triple $[D_j,U_j,D_{j+1}]$ is a special pseudoconvex bump in ${\cal X}$ with $\bar U_j\subset {\cal X}\setminus K$.

Put $P'_j\edf P'|_{\bar D_j\setminus K}$. We have $P'_0\simeq P$ by assumption, and $P'_{j+1}|_{\bar D_j\setminus K}=P'_j$ for $0\leq j\leq k$ by the definition of $P'_j$, $P'_{j+1}$. By induction, applying Proposition \ref{ExtBdlOnBump2} (2) and starting with ${\cal P}_0\edf{\cal P}$, we obtain for $0\leq j\leq k+1$ holomorphic bundles ${\cal P}_j$ of class $\rg$ on $\bar D_j\setminus K$ such that the underlying topological bundle $P_j$ of ${\cal P}_j$ is isomorphic to $P'_j$, and such that ${\cal P}_{j+1}|_{\bar D_j\setminus K}={\cal P}_j$ for $j\leq k$.
 At the $j$-th step, when we construct ${\cal P}_{j+1}$ by extending ${\cal P}_j$, we do  have a topological bundle  on $\bar D_{j+1}\setminus K$ whose restriction to $\bar D_{j}\setminus K$ is isomorphic to $P_j$, namely $P'_{j+1}$; so Proposition \ref{ExtBdlOnBump2} (2) applies.
 
 In particular,  we obtain a holomorphic bundle ${\cal Q}\edf {\cal P}_{k+1}$ of class $\rg$ on $\bar D_{k+1}\setminus K=Z\setminus K$ whose restriction to $\bar D_{0}\setminus K=X\setminus K$ is   ${\cal P}$ and whose underlying topological bundle is isomorphic to $P'_{k+1}=P'$.

%
%
%

%

\end{proof}

\end{document}